\documentclass[11pt,reqno]{amsart}
\usepackage{ytableau, tikz, hyperref}
\usepackage{palatino}
\usepackage{mathtools}
\usepackage[left=3cm,right=3cm,top=3cm,bottom=3cm]{geometry}
\usepackage[shortlabels,inline]{enumitem}
\newcommand\characx{\mathbf{x}}
\newcommand{\N}{\mathcal{N}}

\newtheorem{theorem}{Theorem}
\newtheorem{remark}[theorem]{Remark}
\newtheorem{example}[theorem]{Example}
\newcommand\bremark{\begin{remark}\begin{upshape}}
\newcommand\eremark{\end{upshape}\end{remark}}
\newtheorem{proposition}[theorem]{Proposition}
\newtheorem{corollary}[theorem]{Corollary}

\newtheorem{definition}[theorem]{Definition}
\DeclareMathOperator{\Tab}{Tab}
\DeclareMathOperator{\wt}{wt}
\DeclareMathOperator{\Cont}{Cont}

\DeclareMathOperator{\GT}{GT}
\DeclareMathOperator{\MGT}{MGT}
\DeclareMathOperator{\SVT}{SVT}

\DeclareMathOperator{\SVCT}{SVCT}
\newcommand\encircle[1]{%
  \tikz[baseline=(X.base)] 
    \node (X) [draw, shape=circle, inner sep=0] {\strut #1};}
\allowdisplaybreaks
\title[A contratableau model for K-theoretic Littlewood--Richardson rule]{A contratableau model for K-theoretic Littlewood--Richardson rule}
\author[]{Siddheswar Kundu}
\address{Department of Mathematics, Indian Institute of Technology Kanpur, Kanpur 208016, India.}
\email{kundusidhu96@gmail.com}
\begin{document}
\begin{abstract}
The K-theoretic Littlewood--Richardson rule, established by A. Buch, is a combinatorial method for counting the coefficients in the expansion of the product of two symmetric Grothendieck polynomials as a linear combination of symmetric Grothendieck polynomials. In this paper, we provide an explicit combinatorial formula in terms of set-valued contratableaux for the K-theoretic Littlewood--Richardson rule, generalizing the contratableau model for the classical Littlewood--Richardson rule introduced by Carr{\'e}.
\end{abstract}
\maketitle
\section{Introduction}
Grothedieck polynomials were defined by Lascoux and Schützenberger, and they provide formulas for the structural sheaves of the Schubert varieties in a flag variety \cite{Lascoux:G1}. These polynomials were further understood combinatorially by Fomin and Kirillov \cite{Fomin:Yang-Baxter}. They are indexed by permutations in the symmetric group $S_n$, as in the case of Schubert polynomials and when the stable limit of $n \rightarrow \infty$ is taken into account, Grothendieck polynomials become symmetric functions. This paper explores symmetric Grothendieck polynomials $G_{\lambda}$ for a partition $\lambda,$ which are stable Grothendieck polynomials associated to Grassmannian permutation $w_{\lambda},$ (see e.g., \cite[\S2]{Buch:K-LR} for more details about $w_{\lambda}$). Buch \cite{Buch:K-LR} proved the following (see \S\ref{Section 2} for the notations) 
$$G_{\lambda}(\characx)= \displaystyle \sum _{T \in \SVT_n(\lambda)} (-1)^{|T|-|\lambda|}\characx^{\wt(T)}.$$ 
Thus $G_{\lambda}$ can be considered as a K-theory analogue of the Schur function $s_{\lambda}$ as $\{G_{\lambda}(\characx)\}$ indexed by partitions is a basis for (a completion of) the space of symmetric functions, see \cite{Lenart:Comb-aspects}. For $\lambda, \mu, \nu \in \mathcal{P}[n]$, the K-theoretic Littlewood--Richardson coefficients $C^{\nu}_{\lambda,\mu}$ are defined as follows:
$$ G_{\lambda}(\characx)G_{\mu}(\characx)=\displaystyle\sum_{\nu \in \mathcal{P}[n]}(-1)^{|\nu|-|\lambda|-|\mu|} C^{\nu}_{\lambda,\mu} G_{\nu}(\characx).$$
The coefficients $C^{\nu}_{\lambda,\mu}$ are non-zero only if $|\lambda|+ |\mu| \leq |\nu|$. When $|\lambda|+|\mu|=|\nu|,$ the coefficients $C^{\nu}_{\lambda,\mu}$ are the classical Littlewood--Richardson coefficients $c^{\nu}_{\lambda,\mu}$, defined by
$$ s_{\lambda}(\characx)s_{\mu}(\characx)=\displaystyle\sum_{\nu \in \mathcal{P}[n]}c^{\nu}_{\lambda,\mu} s_{\nu}(\characx).$$
Buch \cite[Theorem 5.4]{Buch:K-LR} provided a rule to count the coefficients $C^{\nu}_{\lambda,\mu}$, which reduce to the usual Littlewood--Richardson rule when $|\nu|=|\lambda|+|\mu|$. Buch's rule is further proved in \cite{K-Bender-Knuth} using Bender--Kunth-type involutions. The main theorem of this article is to present a new rule (Theorem \ref{Theorem:main}) for the coefficients $C_{\lambda,\mu}^{\nu}$ involving set-valued contratableaux \S\ref{Section 3}, which extends the similar rule for $c_{\lambda,\mu}^{\nu}$ \cite{Carre:LR}.

In \cite{Knutson-Tao:saturation}, Knutson and Tao gave a proof of the saturation conjecture for $c_{\lambda,\mu}^{\nu}$, utilizing two new characterizations of Berenstein--Zelevinsky polytopes, referred to as honeycomb models and hive models. Then, using the hive model, Buch \cite {Buch:saturation} provided a simple proof of this result. In \cite[Appendix A]{Buch:saturation} Buch gave a simple and direct bijection between the hives with certain boundary, given by partitions $\lambda, \mu, \nu$, and the set of Littlewood--Richardson skew tableaux of shape $\nu/\lambda$ and weight $\mu$. In the proof, it was shown that those hives and the set of all $\mu$-dominant contratableaux of shape $\lambda$ with weight $\nu-\mu$, have the same cardinality. Our approach to prove the main theorem, i.e., Theorem~\ref{Theorem:main}, by extending this idea but without giving a hive model for $C_{\lambda,\mu}^{\nu}$. In \cite[Theorem 5.4]{Buch:K-LR}, it was proved that $C_{\lambda,\mu}^{\nu}$ counts the cardinality of the set $\SVT^{\lambda}_{\nu-\lambda}(\mu),$ which means, the set of all $\lambda$-dominant semi-standard set-valued tableaux of shape $\mu$ with weight $\nu-\lambda$. In this article, we prove $C_{\lambda,\mu}^{\nu}$ is the number of elements in the set $\SVCT^{\mu}_{\nu-\mu}(\lambda),$ namely, the set of all $\mu$-dominant set-valued contratableaux of shape  $\lambda$ with weight $\nu-\mu$. To prove this, we establish a direct bijection between $\SVT^{\lambda}_{\nu-\lambda}(\mu)$ and $\SVCT^{\mu}_{\nu-\mu}(\lambda)$ without recourse to the hive model construction for $C_{\lambda,\mu}^{\nu}$.  
\section{Preliminaries}
\label{Section 2}
\subsection{Partitions and Young diagrams}
We set $\mathbb{Z}_{+} =\{ 0,1,2,\dots \}$ and $\mathbb{N} =\{ 1,2,\dots \}$.

A \emph{partition} $\lambda=(\lambda_1,\ldots,\lambda_l),$ is a non-negative integer sequence such that $\lambda_1 \geq \cdots \geq \lambda_l \geq 0.$ We define the \emph{length} of $\lambda$ to be the smallest integer $r$ such that $\lambda_r > 0$ and $\lambda_{r+1} =0$. We write $r=l(\lambda)$ and $|\lambda|=\lambda_1 +\cdots +\lambda_l $. We set $\mathcal{P}[n]$ as the set of all partitions with length at most $n$. For a partition $\lambda$, the set $\{(i,j)\in \mathbb{N} \times \mathbb{N} : 1\leq i \leq l(\lambda), 1\leq j \leq \lambda_i\}$ is called \emph{Young diagram} of $\lambda$. We use the notation that the Young diagram of $\lambda$ is the diagram obtained by arraying $l$ boxes having $l(\lambda)$ left-justified rows with the $i^{th}$ row consisting $\lambda_i$ boxes. Throughout the paper we often make no distinction between partitions and the corresponding Young diagrams. We say that $\mu \subset \lambda $ if $\mu_i \leq \lambda_i$ for all $i>0$. A \emph{skew Young diagram} $\lambda/\mu$ is defined to be the set-theoretic difference $\lambda - \mu$ of the Young diagrams, where $\mu \subset \lambda$. 
\subsection{Semi-standard set-valued tableau}Let $[n]=\{ 1,2,\dots,n\}$. A \emph{filling} of a skew Young diagram $\lambda/\mu$ is a map from the set of all boxes in $\lambda/\mu$ to the set of non-empty subsets of $[n]$. We define a \emph{semi-standard set-valued tableau} of shape $\lambda/\mu$ to be a filling of the skew Young diagram $\lambda/\mu$, such that the rows are weakly increasing from left to right and the columns are strictly increasing from top to bottom in the following manner: whenever a set $A$ precedes another set $B$ in a row, $\max A \leq \min B$; and whenever $C$ lies below $A$ in a column, $\max A < \min C$. We simply write a set-valued tableau to refer a semi-standard set-valued tableau.

Given a skew diagram $\lambda/\mu ,\SVT_{n}(\lambda/\mu)$ is the set of all set-valued tableaux of shape $\lambda/\mu$ with entries $\leq n$. The \emph{weight} of a set-valued tableau $T \in \SVT_{n}(\lambda/\mu)$, denoted by $\wt(T)$, is the $n$-tuple $(t_1,\ldots,t_n)$ such that $t_i$ is the number of occurrences $i$ in $T$.
\begin{example}  
$T=  \ytableausetup{mathmode,
notabloids,boxsize=1.5em}
 \begin{ytableau}
 \none & \none &1,\!2&2,\!3\\
 \none &1 &3,\!4 \\
 1,\!4 & 4
  \end{ytableau}
 \in \SVT_{4}((4,3,2)/(2,1))$ with $\wt(T)=(3,2,2,3)$.
\end{example}

A \emph{semi-standard Young tableau} of shape $\lambda/\mu$ is a semi-standard set-valued tableau of shape $\lambda/\mu$
where each box of $\lambda /\mu$ is filled by a positive integer in $[n]$.
We let $\Tab_n(\lambda/\mu)$ denote the set of all semi-standard Young tableaux in $\SVT_n(\lambda/\mu)$. 
\subsection{Skew symmetric Grothendieck polynomial}We define the \emph{skew symmetric Grothendieck polynomial} $G_{\lambda/\mu}(\mathbf{x})$ by 
$$ G_{\lambda/\mu}( \mathbf{x}):= \displaystyle \sum _{T \in \SVT_n(\lambda/\mu)} (-1)^{|T|-|\lambda|+|\mu|}\characx^{\wt(T)},$$ where
for $\alpha =(\alpha_1,\alpha_2,\ldots, \alpha_n)\in \mathbb{Z}^n _{+} ,\text{ we let }  \characx^{\alpha}=x_1^{\alpha_1}x_2^{\alpha_2}\cdots x_n^{\alpha_n},$ and $|T|$ is the total number of entries in it.

The lowest degree homogeneous component of $G_{\lambda/\mu}(\mathbf{x})$ is the skew Schur polynomial $s_{\lambda/\mu}(\characx)$, which is defined by 
$$ s_{\lambda/\mu}(\characx):= \displaystyle \sum _{T \in \Tab_n(\lambda/\mu)} \characx^{\wt(T)}.$$
\section{Set-valued contratableau}
\label{Section 3}
In this section we define set-valued contratableau and each set-valued contratableau corresponds to a unique marked Gelfand--Tsetlin (GT) pattern, which are discussed in \cite[\S4.2]{Travis:symmetric}.
\begin{definition}
For a given $\lambda,$ the skew shape $C(\lambda)$ is defined by rotating $Y(\lambda)$ 180 degrees, so that the new diagram has $\lambda_i$ boxes in $i^{th}$ row from the bottom and the rows are right justified. A set-valued contratableau of shape $\lambda$ is a semi-standard set-valued tableau of shape $C (\lambda)$. For $\lambda \in \mathcal{P}[n]$, $\SVCT_m(\lambda)$ is the set of all set-valued contratableaux of shape $\lambda$ with entries $\leq m$.
\end{definition}
\begin{remark}
 A contratableau of shape $\lambda$ is a semi-standard Young tableau of shape $C (\lambda)$, see \cite[\S1]{Carre:LR} \cite[Appendix]{Buch:saturation}. For $\lambda \in \mathcal{P}[n]$, $\Cont_m(\lambda)$ is the set of all contratableau of shape $\lambda$ with entries $\leq m$.
 \end{remark}
\begin{example}
    For $\lambda=(4,2,1), C(\lambda)=$
\ytableausetup{nosmalltableaux}
\begin{ytableau} 
\none & \none & \none & \null \\
\none & \none & \null & \null \\
\null & \null & \null & \null \\
\end{ytableau}
and  
\ytableausetup{mathmode,
notabloids}
\begin{ytableau}
 \none & \none & \none & 1,\!2\\
 \none & \none & 2,\!3&3 \\
  1 & 1,\!3 & 4 &4
\end{ytableau} is a set-valued contratableau of shape $\lambda$.
\end{example}
\subsection{GT patterns} A GT pattern of size $n$ is a triangular array of integers $X=(x_{i,j})_{1\leq j \leq i \leq n}$ (see Figure~\ref{figure:array}) satisfying the ``North-East'' (NE), ''South-East'' (SE) inequalities given below:
\begin{center}
$NE_{i,j}(X)= x_{i,j} - x_{(i-1),j} \geq 0 \quad \quad \quad 1 \leq j < i \leq n,$\\ 
$SE_{i,j}(X)=  x_{(i-1),j} - x_{i,(j+1)} \geq 0 \quad \quad 1 \leq j < i \leq n $.
\end{center}
Given $\lambda \in \mathcal{P}[n],$ we use the notation $\GT_{\mathbb{Z}}(\lambda)$ to denote the set of all GT patterns $X=(x_{i,j})_{1\leq j \leq i \leq n}$ such that
$x_{nj}=\lambda_{j}$ for $1 \leq j \leq n$. Let $X \in \GT_{\mathbb{Z}}(\lambda)$. Then the $i$-tuple $x^{(i)}$ defined by $x^{(i)}:=(x_{i,1},x_{i,2},\ldots, x_{i,i})$ is a partition with $l(x^{(i)}) \leq i$ for $1 \leq i \leq n$. Also, the skew shape $x^{(i)}/x^{(i-1)}$ $(x^{(0)}: = \emptyset)$ is a horizontal strip, i.e., it does not contain a vertical domino. Thus $X \in \GT_{\mathbb{Z}}(\lambda)$ if and only if $x^{(i)}/x^{(i-1)}$ is a horizontal strip for $1 \leq i \leq n$.
\begin{figure}
    \centering
    \begin{tikzpicture}[scale=1]
	\draw (1.5+0.4,1.5*1.732) node {$x_{1,1}$};
	\draw (1+0.3,1.732) node {$x_{2,1}$};
	\draw (2+0.7,1.732) node {$x_{2,2}$};
	\draw (0.5+0.2,0.5*1.732) node {$x_{3,1}$};
	\draw (1.5+0.6,0.5*1.732) node {$x_{3,2}$};
	\draw (2.5+1,0.5*1.732) node {$x_{3,3}$};
	\draw (0,0) node {$x_{4,1}$};
	\draw (1.4,0) node {$x_{4,2}$};
	\draw (2.8,0) node {$x_{4,3}$};
	\draw (4.2,0) node {$x_{4,4}$};
\end{tikzpicture}
    \caption{A Gelfand--Tsetlin array for $n=4$}
    \label{figure:array}
\end{figure}

\begin{example}
\label{example:GT}
An example of a GT pattern in $\GT_{\mathbb{Z}}(4,3,2,0)$ is given below:
$$
\begin{tikzpicture}[scale=1]
	\draw (1.5,1.5*1.732) node {$2$};
	\draw (1,1.732) node {$3$};
	\draw (2,1.732) node {$1$};
	\draw (0.5,0.5*1.732) node {$3$};
	\draw (1.5,0.5*1.732) node {$2$};
	\draw (2.5,0.5*1.732) node {$1$};
	\draw (0,0) node {$4$};
	\draw (1,0) node {$3$};
	\draw (2,0) node {$2$};
	\draw (3,0) node {$0$};
\end{tikzpicture}.
$$
\end{example}
\begin{definition}[{\cite[Definition 4.3]{Travis:symmetric}}]
\label{def:MGT}
A marked GT pattern of size $n$ is a pair $(X,M)$, where $X=(x_{i,j})_{1\leq j \leq i \leq n}$ is a GT pattern of size $n$ together with a set $M$ of entries that are $``\text{marked}"$, where $M$ is a subset of the set $\{ (i,j) : 1 \leq j < i \leq n  \text{ and } SE_{i,j}(X)>0\}$.
Given a GT pattern $X$, $\MGT(X)$ is the set of all marked GT patterns whose corresponding GT pattern $X$, together with $X$. Given $\lambda \in \mathcal{P}[n],$ we define $\MGT_{\mathbb{Z}}(\lambda):=\displaystyle\bigcup_{X \in \GT_{\mathbb{Z}}(\lambda)} \MGT(X)$. Clearly, $\GT_{\mathbb{Z}}(\lambda) \text{ is a subset of }\MGT_{\mathbb{Z}}(\lambda)$.
\end{definition}
\begin{example}
Consider the GT pattern $X$ in Example~\ref{example:GT}. Then $\MGT_{\mathbb{Z}}(X)$ contains the following marked GT patterns:
\begin{center}
\begin{tikzpicture}[scale=0.8]
	\draw (1.5,1.5*1.732) node {$2$};
	\draw (1,1.732) node {$3$};
	\draw (2,1.732) node {$1$};
	\draw (0.5,0.5*1.732) node {$3$};
	\draw (1.5,0.5*1.732) node {$2$};
	\draw (2.5,0.5*1.732) node {$1$};
	\draw (0,0) node {$4$};
	\draw (1,0) node {$3$};
	\draw (2,0) node {$2$};
	\draw (3,0) node {$0$};
\end{tikzpicture}
\begin{tikzpicture}
    \node at (0,0)   {$\null$};
    \node at (0.5,0) {$\null$};
\end{tikzpicture}
\begin{tikzpicture}[scale=0.8]
	\draw (1.5,1.5*1.732) node {$2$};
	\draw (1,1.732) node {$3$};
	\draw (2,1.732) node {$1$};
	\draw (0.5,0.5*1.732) node {$\encircle{3}$};
	\draw (1.5,0.5*1.732) node {$2$};
	\draw (2.5,0.5*1.732) node {$1$};
	\draw (0,0) node {$4$};
	\draw (1,0) node {$3$};
	\draw (2,0) node {$2$};
	\draw (3,0) node {$0$};
\end{tikzpicture}
\begin{tikzpicture}
    \node at (0,0)   {$\null$};
    \node at (0.5,0) {$\null$};
\end{tikzpicture}
\begin{tikzpicture}[scale=0.8]
	\draw (1.5,1.5*1.732) node {$2$};
	\draw (1,1.732) node {$3$};
	\draw (2,1.732) node {$1$};
	\draw (0.5,0.5*1.732) node {$\encircle{3}$};
	\draw (1.5,0.5*1.732) node {$2$};
	\draw (2.5,0.5*1.732) node {$1$};
	\draw (0,0) node {$4$};
	\draw (1,0) node {$3$};
	\draw (2,0) node {$\encircle{2}$};
	\draw (3,0) node {$0$};
\end{tikzpicture}
\begin{tikzpicture}
    \node at (0,0)   {$\null$};
    \node at (0.5,0) {$\null$};
\end{tikzpicture}
\begin{tikzpicture}[scale=0.8]
	\draw (1.5,1.5*1.732) node {$2$};
	\draw (1,1.732) node {$\encircle{3}$};
	\draw (2,1.732) node {$1$};
	\draw (0.5,0.5*1.732) node {$\encircle{3}$};
	\draw (1.5,0.5*1.732) node {$2$};
	\draw (2.5,0.5*1.732) node {$1$};
	\draw (0,0) node {$4$};
	\draw (1,0) node {$3$};
	\draw (2,0) node {$2$};
	\draw (3,0) node {$0$};
\end{tikzpicture}

\begin{tikzpicture}[scale=0.8]
	\draw (1.5,1.5*1.732) node {$2$};
	\draw (1,1.732) node {$\encircle{3}$};
	\draw (2,1.732) node {$1$};
	\draw (0.5,0.5*1.732) node {$3$};
	\draw (1.5,0.5*1.732) node {$2$};
	\draw (2.5,0.5*1.732) node {$1$};
	\draw (0,0) node {$4$};
	\draw (1,0) node {$3$};
	\draw (2,0) node {$2$};
	\draw (3,0) node {$0$};
\end{tikzpicture}
\begin{tikzpicture}
    \node at (0,0)   {$\null$};
    \node at (0.5,0) {$\null$};
\end{tikzpicture}
\begin{tikzpicture}[scale=0.8]
	\draw (1.5,1.5*1.732) node {$2$};
	\draw (1,1.732) node {$\encircle{3}$};
	\draw (2,1.732) node {$1$};
	\draw (0.5,0.5*1.732) node {$3$};
	\draw (1.5,0.5*1.732) node {$2$};
	\draw (2.5,0.5*1.732) node {$1$};
	\draw (0,0) node {$4$};
	\draw (1,0) node {$3$};
	\draw (2,0) node {$\encircle{2}$};
	\draw (3,0) node {$0$};
\end{tikzpicture}
\begin{tikzpicture}
    \node at (0,0)   {$\null$};
    \node at (0.5,0) {$\null$};
\end{tikzpicture}
\begin{tikzpicture}[scale=0.8]
	\draw (1.5,1.5*1.732) node {$2$};
	\draw (1,1.732) node {$3$};
	\draw (2,1.732) node {$1$};
	\draw (0.5,0.5*1.732) node {$3$};
	\draw (1.5,0.5*1.732) node {$2$};
	\draw (2.5,0.5*1.732) node {$1$};
	\draw (0,0) node {$4$};
	\draw (1,0) node {$3$};
	\draw (2,0) node {$\encircle{2}$};
	\draw (3,0) node {$0$};
\end{tikzpicture}
\begin{tikzpicture}
    \node at (0,0)   {$\null$};
    \node at (0.5,0) {$\null$};
\end{tikzpicture}
\begin{tikzpicture}[scale=0.8]
	\draw (1.5,1.5*1.732) node {$2$};
	\draw (1,1.732) node {$\encircle{3}$};
	\draw (2,1.732) node {$1$};
	\draw (0.5,0.5*1.732) node {$\encircle{3}$};
	\draw (1.5,0.5*1.732) node {$2$};
	\draw (2.5,0.5*1.732) node {$1$};
	\draw (0,0) node {$4$};
	\draw (1,0) node {$3$};
	\draw (2,0) node {$\encircle{2}$};
	\draw (3,0) node {$0$};
\end{tikzpicture}.
\end{center}
\end{example}
For $\lambda \in \mathcal{P}[n],$ we recall the bijection $ \Upsilon: \MGT_{\mathbb{Z}}(\lambda) \rightarrow \SVT_n(\lambda)$ in \cite[Proposition 5]{Travis:symmetric}. Let $(X,M) \in \MGT_{\mathbb{Z}}(\lambda)$. We construct $\Upsilon(X,M)$ recursively. Let us assume we have added all the entries $1,2,\dots,i-1$ and the result is $X^{(i-1)}$. Now we fill each box of the horizontal strip $x^{(i)}/x^{(i-1)}$ with an $i$ and add to $X^{(i-1)}$.
In addition, if $(i,j) \in M$ then add an $`i'$ to the rightmost box containing $(i-1)$ in $j^{th}$ row of $ X^{(i-1)}$ and we obtain $X^{(i)}$. We note that this is the unique position where we can add $i$ to the $j^{th}$ row of $X^{(i-1)}$. Finally, we define $\Upsilon(X,M):=X^{(n)}$. Clearly, the process is reversible. Thus we have the following the corollary.
\begin{corollary}
    The bijection $\Upsilon: \MGT_{\mathbb{Z}}(\lambda) \rightarrow \SVT_n(\lambda)$ restricts to a bijection between $\GT_{\mathbb{Z}}(\lambda)$ and $\Tab_n(\lambda)$.
\end{corollary}
\begin{example}
Consider the marked GT pattern $(X,M)$ given by:
$$
\begin{tikzpicture}[scale=1]
\draw (1.5,1.5*1.732) node {$2$};
\draw (1,1.732) node {$3$};
\draw (2,1.732) node {$1$};
\draw (0.5,0.5*1.732) node {$\encircle{3}$};
\draw (1.5,0.5*1.732) node {$2$};
\draw (2.5,0.5*1.732) node {$1$};
\draw (0,0) node {$4$};
\draw (1,0) node {$3$};
\draw (2,0) node {$\encircle{2}$};
\draw (3,0) node {$0$};
\end{tikzpicture}.
$$
Then, the sequence $(X^{(1)}, X^{(2)}, X^{(3)}, X^{(4)})$ is given below:
$$
\begin{tikzpicture}[scale=1.2]
  \draw (0,0)--(0,0.5)--(1,0.5)--(1,0)--(0,0);  
  \draw (0.5,0)--(0.5,0.5);
  \draw (0.25,0.25) node {$1$};
  \draw (0.75,0.25) node {$1$};
  \draw (0.5,-0.5) node {$X^{(1)}$};
\end{tikzpicture}
\begin{tikzpicture}
    \node at (0,0)   {$\null$};
    \node at (0.5,0) {$\null$};
\end{tikzpicture}
\begin{tikzpicture}[scale=1.2]
  \draw (0,0)--(0,1)--(1.5,1)--(1.5,0.5)--(0.5,0.5)--(0.5,0)--(0,0);  
  \draw (0,0.5)--(0.5,0.5)--(0.5,1);
  \draw (1,0.5)--(1,1);
  \draw (0.25,0.25) node {$2$};
  \draw (0.25,0.75) node {$1$};
  \draw (0.75,0.75) node {$1$};
  \draw (1.25,0.75) node {$2$};
  \draw (0.5,-0.5) node {$X^{(2)}$};
\end{tikzpicture}
\begin{tikzpicture}
    \node at (0,0)   {$\null$};
    \node at (0.5,0) {$\null$};
\end{tikzpicture}
\begin{tikzpicture}[scale=1.2]
  \draw (0,0)--(0,1.5)--(1.5,1.5)--(1.5,1)--(1,1)--(1,0.5)--(0.5,0.5)--(0.5,0)--(0,0);  
  \draw (0,0.5)--(0.5,0.5)--(0.5,1.5);
  \draw (0,1)--(1.5,1);
  \draw (1,1)--(1,1.5);
  \draw (0.25,0.25) node {$3$};
  \draw (0.25,0.75) node {$2$};
  \draw (0.25,1.25) node {$1$};
  \draw (0.75,0.75) node {$3$};
  \draw (0.75,1.25) node {$1$};
  \draw (1.25,1.25-0.02) node {$ 2,\!3$};
  \draw (0.5,-0.5) node {$X^{(3)}$};
\end{tikzpicture}
\begin{tikzpicture}
    \node at (0,0)   {$\null$};
    \node at (0.5,0) {$\null$};
\end{tikzpicture}
\begin{tikzpicture}[scale=1.2]
  \draw (0,0)--(0,1.5)--(2,1.5)--(2,1)--(1.5,1)--(1.5,0.5)--(1,0.5)--(1,0)--(0,0);  
  \draw (0,0.5)--(1.5,0.5);
  \draw (0,1)--(1.5,1);
  \draw (0.5,0)--(0.5,1.5);
  \draw (1,0.5)--(1,1.5);
  \draw (1.5,1)--(1.5,1.5);
  \draw (0.25,0.25-0.03) node {$3,\!4$};
  \draw (0.25,0.75) node {$2$};
  \draw (0.25,1.25) node {$1$};
  \draw (0.75,0.25) node {$4$};
  \draw (0.75,0.75) node {$3$};
  \draw (0.75,1.25) node {$1$};
  \draw (1.25,0.75) node {$4$};
  \draw (1.25,1.25-0.02) node {$2,\!3$};
  \draw (1.75,1.25) node {$4$};
  \draw (0.5,-0.5) node {$X^{(4)}$};
\end{tikzpicture}.
$$
Therefore we obtain $\Upsilon(X,M)=X^{(4)}$.
In the same way, we have:
$$
\begin{tikzpicture}
\draw (1.5,1.5*1.732) node {$2$};
\draw (1,1.732) node {$3$};
\draw (2,1.732) node {$1$};
\draw (0.5,0.5*1.732) node {$3$};
\draw (1.5,0.5*1.732) node {$2$};
\draw (2.5,0.5*1.732) node {$1$};
\draw (0,0) node {$4$};
\draw (1,0) node {$3$};
\draw (2,0) node {$2$};
\draw (3,0) node {$0$};  
\end{tikzpicture}
\begin{tikzpicture}
    \draw (0,0) node {\null};
    \draw[|->] (0.5,1.5) -- (1.5,1.5);
    \draw (1,1.75) node {$\Upsilon$};
    \draw (2.5,1.5) node {\null};
\end{tikzpicture}
\begin{tikzpicture}
    \draw (0,0)--(0,2.4)--(3.2,2.4);
    \draw (0,1.6)--(3.2,1.6);
    \draw (0,0.8)--(2.4,0.8);
    \draw (0,0)--(1.6,0);
    \draw (0.8,0)--(0.8,2.4);
    \draw (1.6,0)--(1.6,2.4);
    \draw (2.4,0.8)--(2.4,2.4);
    \draw (3.2,1.6)--(3.2,2.4);
    \node at (0.4,0.4) {$3$};
    \node at (1.2,0.4) {$4$};
    \node at (0.4,1.2) {$2$};
    \node at (1.2,1.2) {$3$};
    \node at (2,1.2) {$4$};
    \node at (0.4,2) {$1$};
    \node at (1.2,2) {$1$};
    \node at (2,2) {$2$};
    \node at (2.8,2) {$4$};
\end{tikzpicture}.
$$
\end{example}
Let $\lambda,\mu \in \mathcal{P}[n]$ be such that $\mu \subset \lambda$. Then the diagram $C(\lambda) / C(\mu)$ is obtained by removing the boxes of $C(\mu)$ from those of $C(\lambda)$. For $\lambda=(4,3,1), \mu=(2,1),$ $C(\lambda) / C(\mu)$ is the following diagram (omitting the diagram in yellow):
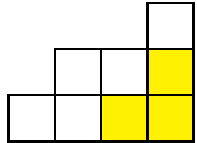
\begin{center}
\ytableausetup{nosmalltableaux,boxsize=0.6 cm}
\begin{ytableau} 
\none & \none & \none & \null  \\
\none & \null & \null & *(yellow) \null\\
\null & \null & *(yellow) \null & *(yellow) \null  \\
\end{ytableau}.
\end{center}
Now given $(Y,M) \in \MGT_{\mathbb{Z}}(\lambda)$, we construct an element in $\SVCT_n(\lambda)$ recursively. Suppose we have added all the entries $n,n-1,\ldots,n+1-(i-1)$ and the result is $Y^{(i-1)}$. Then we fill each box in $C(y^{(i)})/C(y^{(i-1)})$ (which obviously does not contain a vertical domino) by $n+1-i$ and add to $Y^{(i-1)}$. In addition, if $(i,j) \in M$ then we add an $n+1-i$ to the leftmost box containing $n+2-i$ in $j^{th}$ row from bottom of $Y^{(i-1)}$ and we obtain $Y^{(i)}$. Finally we get $Y^{(n)} \in \SVCT_n(\lambda)$. It is evident that this procedure can be reversed. So we obtain the following proposition.
\begin{proposition}
\label{proposition:contretableau}
    The map $\Omega : \MGT_{\mathbb{Z}}(\lambda) \rightarrow \SVCT_n(\lambda)$ defined by $(Y,M) \mapsto Y^{(n)}$ is a bijection.
\end{proposition}
\begin{corollary}
   The bijection $\Omega: \MGT_{\mathbb{Z}}(\lambda) \rightarrow \SVCT_n(\lambda)$ restricts to a bijection between $\GT_{\mathbb{Z}}(\lambda)$ and $\Cont_n(\lambda)$. 
\end{corollary}
\begin{remark}
We can describe the bijection $\Omega$ as follows: Let $(Y,M) \in \MGT_{\mathbb{Z}}(\lambda)$. Then $\Upsilon(Y,M)\in \SVT_n(\lambda)$. Next, we rotate $\Upsilon(Y,M)$ by 180 degrees and replace each $i$ with $n+1-i$ and finally obtain $\Omega(Y,M)$.
\end{remark}
\begin{example}
Consider the marked GT pattern $(Y,M)$ given below:
$$
\begin{tikzpicture}[scale=0.9]
	\draw (1.5,1.5*1.732) node {$2$};
	\draw (1,1.732) node {$\encircle{3}$};
	\draw (2,1.732) node {$1$};
	\draw (0.5,0.5*1.732) node {$\encircle{3}$};
	\draw (1.5,0.5*1.732) node {$2$};
	\draw (2.5,0.5*1.732) node {$1$};
	\draw (0,0) node {$4$};
	\draw (1,0) node {$3$};
	\draw (2,0) node {$\encircle{2}$};
	\draw (3,0) node {$0$};
\end{tikzpicture}.
$$
Then, the sequence $(Y^{(1)}, Y^{(2)}, Y^{(3)}, Y^{(4)})$ is given below:
$$
\begin{tikzpicture}[scale=1.2]
  \draw (0,0)--(0,0.5)--(1,0.5)--(1,0)--(0,0);  
  \draw (0.5,0)--(0.5,0.5);
  \draw (0.25,0.25) node {$4$};
  \draw (0.75,0.25) node {$4$};
  \draw (0.5,-0.5) node {$Y^{(1)}$};
\end{tikzpicture}
\begin{tikzpicture}
    \node at (0,0)   {$\null$};
    \node at (0.5,0) {$\null$};
\end{tikzpicture}
\begin{tikzpicture}[scale=1.2]
  \draw (0,0)--(0,0.5)--(1,0.5)--(1,0)--(0,0);  
  \draw (0.5,0)--(0.5,0.5);
  \draw (1,0.5)--(1,1)--(1.5,1)--(1.5,0)--(1,0);
  \draw (1,0.5)--(1.5,0.5);
  \draw (0.25,0.25) node {$3$};
  \draw (0.75,0.25-0.02) node {$3,\!4$};
  \draw (1.25,0.25) node {$4$};
  \draw (1.25,0.75) node {$3$};
  \draw (0.5+0.3,-0.5) node {$Y^{(2)}$};
\end{tikzpicture}
\begin{tikzpicture}
    \node at (0,0)   {$\null$};
    \node at (0.5,0) {$\null$};
\end{tikzpicture}
\begin{tikzpicture}[scale=1.2]
  \draw (0,0)--(0,0.5)--(1,0.5)--(1,0)--(0,0);  
  \draw (0.5,0)--(0.5,0.5);
  \draw (1,0.5)--(1,1)--(1.5,1)--(1.5,0)--(1,0);
  \draw (1,0.5)--(1.5,0.5);
  \draw (0.5,0.5)--(0.5,1)--(1,1)--(1,1.5)--(1.5,1.5)--(1.5,1);
  \draw (0.25,0.25-0.02) node {$2,\!3$};
  \draw (0.75,0.25-0.02) node {$3,\!4$};
  \draw (1.25,0.25) node {$4$};
  \draw (1.25,0.75) node {$3$};
  \draw (0.75,0.75) node {$2$};
  \draw (1.25,1.25) node {$2$};
  \draw (0.5+0.3,-0.5) node {$Y^{(3)}$};
\end{tikzpicture}
\begin{tikzpicture}
    \node at (0,0)   {$\null$};
    \node at (0.5,0) {$\null$};
\end{tikzpicture}
\begin{tikzpicture}[scale=1.2]
  \draw (0,0)--(0,0.5)--(1,0.5)--(1,0)--(0,0);  
  \draw (0.5,0)--(0.5,0.5);
  \draw (1,0.5)--(1,1)--(1.5,1)--(1.5,0)--(1,0);
  \draw (1,0.5)--(1.5,0.5);
  \draw (0.5,0.5)--(0.5,1)--(1,1)--(1,1.5)--(1.5,1.5)--(1.5,1);
  \draw (0,0)--(-0.5,0)--(-0.5,0.5)--(0,0.5)--(0,1)--(0.5,1)--(0.5,1.5)--(1,1.5);
  \draw (-0.25,0.25) node {$1$};
  \draw (0.25,0.25-0.02) node {$2,\!3$};
  \draw (0.75,0.25-0.02) node {$3,\!4$};
  \draw (1.25,0.25) node {$4$};
  \draw (1.25,0.75) node {$3$};
  \draw (0.25,0.75) node {$1$};
  \draw (0.75,0.75) node {$2$};
  \draw (0.75,1.25) node {$1$};
  \draw (1.25,1.25-0.02) node {$1,\!2$};
  \draw (0.75,-0.5) node {$Y^{(4)}$};
\end{tikzpicture}.
$$
Thus, we have $\Omega(Y,M)=Y^{(4)}$.
Similarly, we obtain the following:
$$
\begin{tikzpicture}[scale=0.9]
	\draw (1.5,1.5*1.732) node {$2$};
	\draw (1,1.732) node {$3$};
	\draw (2,1.732) node {$1$};
	\draw (0.5,0.5*1.732) node {$3$};
	\draw (1.5,0.5*1.732) node {$2$};
	\draw (2.5,0.5*1.732) node {$1$};
	\draw (0,0) node {$4$};
	\draw (1,0) node {$3$};
	\draw (2,0) node {$2$};
	\draw (3,0) node {$0$};
\end{tikzpicture}
\begin{tikzpicture}
\draw (0,0) node {\null};
\draw[|->] (0.5-0.2,1.5) -- (1.5-0.2,1.5);
\draw (1-0.2,1.75) node {$\Omega$};
\draw (2.5-0.7,1.5) node {\null};
\end{tikzpicture}
\begin{tikzpicture}[scale=1.2]
  \draw (0,0)--(0,0.5)--(1,0.5)--(1,0)--(0,0);  
  \draw (0.5,0)--(0.5,0.5);
  \draw (1,0.5)--(1,1)--(1.5,1)--(1.5,0)--(1,0);
  \draw (1,0.5)--(1.5,0.5);
  \draw (0.5,0.5)--(0.5,1)--(1,1)--(1,1.5)--(1.5,1.5)--(1.5,1);
  \draw (0,0)--(-0.5,0)--(-0.5,0.5)--(0,0.5)--(0,1)--(0.5,1)--(0.5,1.5)--(1,1.5);
  \draw (-0.25,0.25) node {$1$};
  \draw (0.25,0.25) node {$3$};
  \draw (0.75,0.25) node {$4$};
  \draw (1.25,0.25) node {$4$};
  \draw (1.25,0.75) node {$3$};
  \draw (0.25,0.75) node {$1$};
  \draw (0.75,0.75) node {$2$};
  \draw (0.75,1.25) node {$1$};
  \draw (1.25,1.25) node {$2$};
\end{tikzpicture}.
$$
\end{example}
\begin{remark}
For $(X,M) \in \MGT_{\mathbb{Z}}(\lambda)$, if $\wt(\Upsilon(X,M)) = (\alpha_1,\alpha_2,\ldots,\alpha_n)$, then $\wt(\Omega(X,M)) = (\alpha_n,\alpha_{n-1},\ldots,\alpha_1)$.
\end{remark}
\section{proof of the main theorem}
\label{Section 4}
In this section, we prove our main theorem (Theorem~\ref{Theorem:main}).
\begin{definition}
 The \emph{column word} $c(T)$ of a set-valued tableau $T$ of skew shape $\theta$ is the word obtained by reading each column of $T$, starting from the rightmost column, according to the following process, and then moving to the left. In each column, we read the entries from top to bottom and within each cell we read the entries in decreasing order.
\end{definition}
\begin{definition}
 The \emph{row word} $r(T)$ of a set-valued tableau $T$ of skew shape $\theta$ is the word obtained by reading each row of $T$, starting from the top row, according to the following procedure, and then continuing down the rows. In each row, we read the entries from right to left and within each cell we read the entries in decreasing order.
\end{definition}
\begin{example}
\label{example:reading word}
$T=  \ytableausetup{mathmode,
notabloids,boxsize=2.5em}
  \begin{ytableau}
    1&1,\!2&2,\!3\\2,\!3,\!4&4   
  \end{ytableau}
$ is a set-valued tableau of shape $(3,2,0)$ with $c(T)=322141432$ and $r(T)=322114432$.
\end{example}
\begin{definition}
    A word $u=u_1u_2 \cdots u_s$ is said to be \emph{dominant word}~\cite[\S5.2]{Fulton:yt} if for all $t \geq 1,$ the number of $i'$s in $ u_1u_2 \cdots u_t$ is at least the number of $(i+1)'$s in it for all $i \geq 1$.
    Also, the word $u$ is said to \emph{$\lambda$-dominant} if the concatenation word $r(T_{\lambda})*u$ is dominant, where $T_{\lambda}$ is the unique semi-standard Young tableau of shape and weight both equals to $\lambda$. Furthermore,
    A set-valued tableau $T$ of skew shape $\theta$ is said to be \emph{$\lambda$-dominant} if $r(T)$ is $\lambda$-dominant.
    
\end{definition}
\begin{example}
The tableau in Example~\ref{example:reading word} is $(4,2,1)$-dominant but not $(3,1)$-dominant. 
\end{example}
Using similar argument in \cite[Proposition 9.5]{Mrigendra:adv:arxiv}, we can state the following.
\begin{proposition}
    Let $T$ be a set-valued tableau of any skew shape and $\lambda$ be a partition. Then $r(T)$ is $\lambda$-dominant if and only if $c(T)$ is $\lambda$-dominant.
\end{proposition}
We now define $\SVT^{\lambda}_{\nu-\lambda}(\mu)$ to be the set of all $\lambda$-dominant semi-standard set-valued tableaux $T$ of shape $\mu$ such that $ \wt(T)= \nu-\lambda$. Then Theorem 5.4 in \cite{Buch:K-LR} can be stated as follows:
\begin{theorem}
$C_{\lambda,\mu}^{\nu}$ is the cardinality of the set $\SVT^{\lambda}_{\nu-\lambda}(\mu)$.
\end{theorem}
The following theorem is the main theorem in this article.
\begin{theorem}
\label{Theorem:main}
    The coefficient $C^{\nu}_{\lambda, \mu}$ is the number of all $\mu$-dominant set-valued contratableaux of shape $\lambda$ with weight $\nu-\mu$.  
\end{theorem}
\begin{proof}
First, we define $\SVCT^{\mu}_{\nu-\mu}(\lambda)$ to be the set of all $\mu$-dominant set-valued contratableaux of shape $\lambda$ satisfying $\wt(S)= \nu-\mu.$ Our approach to prove Theorem~\ref{Theorem:main} is to produce a bijection between the sets $\SVT^{\lambda}_{\nu-\lambda}(\mu)$ and $\SVCT^{\mu}_{\nu-\mu}(\lambda)$.

Let $T \in \SVT^{\lambda}_{\nu-\lambda}(\mu)$ and $\Upsilon^{-1}(T)=(X_T,M_T)$, where $X_T=(x_{i,j})_{1 \leq j \leq i \leq n} \in \GT_{\mathbb{Z}}(\mu)$. Also, let $T_i$ be the $i^{th}$ row of $T$ from the top. For $ 1\leq i \leq j \leq n, 1 \leq k \leq l \leq n ,$ we define 
$$ p^j_i(T_k^l):= \text{the number of occurrences of } i,i+1,\dots,j \text{ appearing in }T_k,T_{k+1},\dots,T_l.$$
For every $1 \leq i \leq n , 0 \leq k \leq i,$ we define 
\begin{equation}
\label{eq:nik}
   \N_{i,k}(T) = \left\{ 
    \begin{array}{ll}
        \lambda_i + p^i_i(T_1^k) & \quad \text{if} \quad k \geq 1,  \\
        \lambda_i & \quad \text{if} \quad k=0. 
    \end{array}  \right. 
\end{equation}
That is to say, $\N_{i,k}(T)$ is the sum of $\lambda_i$ and the number of occurrences of the entry $i$ in $T_1,T_2,\dots,T_k$ for $k \geq 1$. We now verify that $T$ is $\lambda$-dominant if and only if
\begin{equation}
\label{eq:domi}
 \N_{i,k}(T) \geq \N_{i+1,k+1}(T),    
\end{equation}
for $ 1 \leq i <n , 0 \leq k \leq i.$

First, we assume $\N_{i,k}(T) \geq \N_{i+1,k+1}(T)$ for $ 1 \leq i <n, 0 \leq k \leq i$. Let $u=r(T_{\lambda})* r(T),$ which we write as the sequence $u_1\cdots u_s u_{s+1} \cdots u_t$. To show $T$ is $\lambda$-dominant, we have to prove $(\alpha^s_1,\dots, \alpha^s_n)$ is a partition for each $1 \leq s \leq t,$ where $\alpha^s_j$ is the number of occurrences of the entry $j$ in $u_1\dots u_s$ for $j \geq 1$. Suppose that the rightmost occurrence of the entry $i+1$ of $u_1\dots u_s$ is contained in $T_l$. Then using \eqref{eq:domi}, we obtain
$$ \alpha^s_i \geq \N_{i,l-1}(T) \geq \N_{i+1, l}(T) \geq \alpha^s_{i+1}.$$ The first inequality, $\alpha^s_i \geq \N_{i,l-1}(T) $ follows from the fact that $\alpha^s_i$ is the sum of $\N_{i,l-1}(T)$ and the number of occurrences of the entry $i$ in $u_1\dots u_s$ that contained in $T_l,\dots,T_{l'}$ such that $u_s \in T_{l'}$ and the last inequality follows immediately. 

Conversely, let $T$ be $\lambda$-dominant. We fix $i,k$ with $1 \leq i <n , 0 \leq k \leq i.$ Suppose $i+1$ does not appear in $T_1,\dots,T_{k+1},$ then $$\N_{i,k}(T) \geq \N_{i,0}(T)=\lambda_i \geq \lambda_{i+1} = \N_{i+1,0}(T)=\N_{i+1,k+1}(T).$$ 
Otherwise, choose the largest index $r \in \{1,2,\dots, k+1\}$ such that $T_r$ contains an $i+1$, then we have
$$\N_{i,k}(T) \geq \N_{i,r-1}(T) \geq \N_{i+1,r}(T)=\N_{i+1,k+1}(T).$$
Consider the triangular array $Y_T=(y_{i,j})_{1 \leq j \leq i \leq n},$ where
\begin{equation}
    \label{eq:yij}
    y_{i,j}:=\N_{n-i+j,n-i}(T).
\end{equation}
Since $1 \leq j \leq i ,$ it follows that
$$1-i \leq j-i \leq 0.$$
Adding $n$ to each part gives
$$ n+1-i \leq n+j-i \leq n.$$
Finally, using $1 \leq i \leq n,$ we conclude that
$$1 \leq n-i+j \leq n.$$
Therefore the array $Y_{T}$ is well-defined. Now we show that $Y_T \in \GT_{\mathbb{Z}}(\lambda)$.
Since $T$ is $\lambda$-dominant, it follows from \eqref{eq:domi} that
$$NE_{i,j}(Y_T)=y_{i,j}-y_{(i-1),j}= \N_{n-i+j,n-i}(T)-\N_{n-i+j+1,n-i+1}(T) \geq
0.$$
Also, we have
$$ SE_{i,j}(Y_T)=y_{(i-1),j}-y_{i, j+1} =\N_{n-i+j+1,n-i+1}(T)-\N_{n-i+j+1,n-i}(T).$$
Using \eqref{eq:nik}, we see that $SE_{i,j}(Y_T)$ counts the number of occurrences of the entry $n-i+j+1$ in $T_{n-i+1},$ and hence it is non-negative. It is also immediate that $y_{n,j}=\N_{j,0}(T)=\lambda_j \forall j$. Thus, $Y_T \in \GT_{\mathbb{Z}}(\lambda)$.

Let $(i,j) \in M_T$. Then $SE_{n-j+1,i-j}(Y_T)=y_{n-j,i-j}-y_{n-j+1,i-j+1}=\N _{i,j}(T)-N_{i,(j-1)}(T)$. Thus, $SE_{n-j+1,i-j}(Y_T)$ represents the number of occurrences of $i$ in $j^{th}$ row of $T$. Hence, $$SE_{n-j+1,i-j}(Y_T) > 0,$$ since $(i,j) \in M_T$. We now define
\begin{equation}
\label{eq:m't}
M'_T:=\{ (n-j+1,i-j) : (i,j) \in M_T \}.
\end{equation} So $(Y_T,M'_T)\in {\MGT_{\mathbb{Z}}(\lambda)}$. Let $\Tilde{T}=\Omega(Y_T,M_T')$. Our aim is to show that $\Tilde{T} \in \SVCT_{\nu-\mu}^{\mu}(\lambda)$.

Let $\Tilde{T}_i^{\uparrow}$ be the $i^{th}$ row of $\Tilde{T}$ from the bottom. Then for $1 \leq i \leq n, 1\leq k \leq n-i+1$, we define
\begin{equation}
    \label{eq:nik2}
    \N_{i,k}^{\uparrow}(\Tilde{T}) := \mu_i +\text{ the numbers of occurrences of }i \text{ in }\Tilde{T}_k^{\uparrow},\Tilde{T}_{k+1}^{\uparrow},\dots,\Tilde{T}_n^{\uparrow}.
\end{equation}
Since $\Tilde{T}_k^{\uparrow}=\Tilde{T}_{n-k+1},$ the $(n-k+1)^{th}$ row of $\Tilde{T}$ from the top, we can express $\N^{\uparrow}_{i,k}(\Tilde{T})$ as the sum of $\mu_{i}$ and the number of occurrences of the entry $i$ in $\Tilde{T}_1,\dots,\Tilde{T}_{n-k+1}$. It then follows directly from \eqref{eq:domi} that $\Tilde{T}$ is $\mu$-dominant if and only if 
\begin{equation}
    \label{eq:uparrow}
\N^{\uparrow}_{i-1,k+1}(\Tilde{T}) \geq \N^{\uparrow}_{i,k}(\Tilde{T}),
\end{equation}
for $1< i \leq n, 1 \leq k \leq n-i+1$. We now introduce the following definition: for $1 \leq i < n, 1 \leq k \leq n-i,$ 
$$M'_{i,k}:=\{ (n-i+1,k), (n-i+1,k+1),\dots,(n-i+1,n-i) \},$$
and for $1 \leq i \leq n, k=n-i+1,$
$$M'_{i,k}:=\emptyset.$$
Unless otherwise specified, $\N_{i,k}$ and $ \N_{i,k}^{\uparrow}$ serve as shorthand for $\N_{i,k}(T)$ and $ \N_{i,k}^{\uparrow}(\Tilde{T})$ respectively.

We recall that $\Tilde{T}_i^{\uparrow}$ denotes the $i^{th}$ row of $\Tilde{T}$ indexed from the bottom. We then define $\Tilde{T}_{p,[q',q]}^{\uparrow}$ to be the number of occurrences of the entry $p$ within the rows $\Tilde{T}_{q'}^{\uparrow},\Tilde{T}_{q'+1}^{\uparrow},\dots,\Tilde{T}_q^{\uparrow}.$ Due to semi-standard property of the set-valued contratableau $\Tilde{T},$ each entry $j$ can appear only in $\Tilde{T}_{1}^{\uparrow},\Tilde{T}_{2}^{\uparrow},\dots,\Tilde{T}_{n-j+1}^{\uparrow}.$ Thus $\Tilde{T}_{j,[l,l]}^{\uparrow} =0 $ for $l > n-j+1.$ We note that $(n-j+1,n-j+1) \notin M'_T,$ as any $(i',j') \in M'_T$ must satisfy $i' >j'$ (see Definition~\ref{def:MGT}). Defining $y_{n-i,n-i+1}:=0,$ we have
\begin{equation}
    \label{eq:uparrow-0}
\Tilde{T}_{i,[l,l]}^{\uparrow}= (y_{n-i+1,l}-y_{n-i,l}) + |\{(n-i+1,l)\} \cap M'_T|,
\end{equation}
for $1 \leq l \leq n-i+1$. Here $|\{(n-i+1,l)\} \cap M'_T|$ denotes the cardinality of the set $\{(n-i+1,l)\} \cap M'_T$.

Since $\Tilde{T}_{i,[k,n]}^{\uparrow}=\Tilde{T}_{i,[k,k]}^{\uparrow}+\Tilde{T}_{i,[k+1,k+1]}^{\uparrow}+\cdots+\Tilde{T}_{i,[n-i+1,n-i+1]}^{\uparrow},$ and $\N^{\uparrow}_{i,k}=\mu_i + \Tilde{T}_{i,[k,n]}^{\uparrow},$
it follows from \eqref{eq:uparrow-0} that $$\N_{i,k}^{\uparrow}=\mu_i + (y_{n-i+1,k}-y_{n-i,k})+(y_{n-i+1,k+1}-y_{n-i,k+1})+\cdots +(y_{n-i+1,n-i+1}-0) + |M'_{i,k} \cap M'_T|.$$
Next, we define 
\begin{equation}
    \label{eq:mik} 
 M_{i,k} = \left\{ 
    \begin{array}{ll}\{ (k+i,i),(k+i+1,i),\dots, (n,i) \} & \quad \text{if} \quad 1\leq i < n, 1 \leq k \leq n-i,  \\
    \emptyset & \quad \text{if} \quad 1 \leq i \leq n, k=n-i+1. 
    \end{array} \right.
\end{equation}
Then using \eqref{eq:m't} we can conclude that for $1\leq i < n, 1 \leq k \leq n-i ,$
$$ (n-i+1,t) \in M'_{i,k}\cap M'_T  \text{ if and only if }  (t+i,i) \in M_{i,k}\cap M_T,$$
for $ k \leq t \leq n-i$. Thus we obtain $ |M'_{i,k}\cap M'_T|= |M_{i,k}\cap M_T|$ for $1 \leq i \leq n, 1\leq k \leq n-i+1 $. It then follows from \eqref{eq:yij} that
$$\N_{i,k}^{\uparrow}=\mu_i + (\N_{k+i-1,i-1}-\N_{k+i,i})+(\N_{k+i,i-1}-\N_{k+i+1,i})+\cdots +(\N_{n,i-1}-0) + |M_{i,k}\cap M_T|.$$
Thus,
\begin{equation}
 \label{eq:nik-arrow}
 \N_{i,k}^{\uparrow}=\mu_i + (\N_{k+i-1,i-1}+\N_{k+i,i-1}+ \cdots + \N_{n,i-1})-(\N_{k+i,i}+\N_{k+i+1,i}+\cdots +\N_{n,i})+ |M_{i,k}\cap M_T|.
\end{equation}
It is evident that
\begin{equation}
    \label{eq: mu}
    \mu_i=p^n_1(T_i^i) - |M_{i,1}\cap M_T|.
\end{equation}
Also, using \eqref{eq:nik}, we obtain
\begin{equation*}    
 \N_{k+i-1,i-1}+ \cdots + \N_{n,i-1}=( \lambda_{k+i-1}+p_{k+i-1}^{k+i-1}(T_1^{i-1})) +\cdots + (\lambda_n +p_{n}^n(T_1^{i-1})).   
\end{equation*}
Thus we get 
\begin{equation}
\label{eq:A}
     \N_{k+i-1,i-1}+\N_{k+i,i-1}+ \cdots + \N_{n,i-1} =\lambda_{k+i-1}+\lambda_{k+i}+\cdots +\lambda_n+
p_{k+i-1}^n(T_1^{i-1}).
\end{equation}
Similarly, applying \eqref{eq:nik}, we also have
\begin{equation}
    \label{eq:B}
    \N_{k+i,i}+\N_{k+i+1,i}+\cdots +\N_{n,i}=\lambda_{k+i}+\lambda_{k+i+1}+\cdots +\lambda_n + p_{k+i}^n(T_1^{i}).
\end{equation}
Now applying \eqref{eq: mu}, \eqref{eq:A} and \eqref{eq:B} in \eqref{eq:nik-arrow}, we have

$\N_{i,k}^{\uparrow}= p_1^n(T_i^i) - |M_{i,1}\cap M_T| +(\lambda_{k+i-1}+\cdots +\lambda_n)+
p_{k+i-1}^n(T_1^{i-1})-(\lambda_{k+i}+\cdots +\lambda_n) - p_{k+i}^n(T_1^{i}) + |M_{i,k}\cap M_T|.$  

We note the following equalities
\begin{equation}
\label{eq:a}
   p_1^n(T_i^i)=p_1^{k+i-1}(T_i^i)+ p_{k+i}^n(T_i^i), 
\end{equation}
\begin{equation}
\label{eq:b}
    p_{k+i-1}^n(T_1^{i-1})= p_{k+i-1}^{k+i-1}(T_1^{i-1}) + p_{k+i}^n(T_1^{i-1}), 
\end{equation}
\begin{equation}
\label{eq:c}
    p_{k+i}^n(T_1^{i})= p_{k+i}^n(T_1^{i-1})+ p_{k+i}^n(T_i^{i}).
\end{equation}
Applying \eqref{eq:a}, \eqref{eq:b} and \eqref{eq:c} in the above expression of $\N_{i,k}^{\uparrow}$ yields
\begin{equation}
    \label{eq:nik4}
\N_{i,k}^{\uparrow}=\lambda_{k+i-1} + 
p_{1}^{k+i-1}(T_i^{i})+
p_{k+i-1}^{k+i-1}(T_1^{i-1})- |M_{i,1}\cap M_T| +|M_{i,k}\cap M_T|.
\end{equation}
We now compute $\N_{i-1,k+1}^{\uparrow}-\N_{i,k}^{\uparrow}$ using \eqref{eq:nik4} and establish that $\N^{\uparrow}_{i-1,k+1} \geq \N^{\uparrow}_{i,k}$ for $1< i \leq n$ and $ 1 \leq k \leq n-i+1,$ thereby showing that $\Tilde{T}$ is $\mu$-dominant (by \eqref{eq:uparrow}).

$\N_{i-1,k+1}^{\uparrow}-\N_{i,k}^{\uparrow}=p_{1}^{k+i-1}(T_{i-1}^{i-1})+
p_{k+i-1}^{k+i-1}(T_1^{i-2})- |M_{i-1,1}\cap M_T| +|M_{i-1,k+1}\cap M_T|-
p_{1}^{k+i-1}(T_i^{i})-
p_{k+i-1}^{k+i-1}(T_1^{i-1})+ |M_{i,1}\cap M_T| - |M_{i,k}\cap M_T|.$

Additionally, we note that
\begin{equation}
    \label{eq:x}
p_{1}^{k+i-1}(T_{i-1}^{i-1})=p_{1}^{k+i-2}(T_{i-1}^{i-1})+p_{k+i-1}^{k+i-1}(T_{i-1}^{i-1}),  
\end{equation}
\begin{equation}
    \label{eq:y}
p_{k+i-1}^{k+i-1}(T_1^{i-2})=p_{k+i-1}^{k+i-1}(T_1^{i-1})-p_{k+i-1}^{k+i-1}(T_{i-1}^{i-1}).  \end{equation}
Applying \eqref{eq:x} and \eqref{eq:y} in the above expression of $ \N_{i-1,k+1}^{\uparrow}-\N_{i,k}^{\uparrow},$ we get

$\N_{i-1,k+1}^{\uparrow}-\N_{i,k}^{\uparrow}=\Bigl\{p_{1}^{k+i-2}(T_{i-1}^{i-1}) 
-|M_{i-1,1}\cap M_T| +|M_{i-1,k+1}\cap M_T| \Bigl\}- \Bigl\{ p_{1}^{k+i-1}(T_i^i) -|M_{i,1}\cap M_T| +|M_{i,k}\cap M_T| \Bigl\}.$

We note that $p_1^i(T_j^j)=x_{i,j} +|\{ (j+1,j),(j+2,j),\dots, (i,j)\} \cap M_T |.$ Hence using \eqref{eq:mik}, we have
$$x_{i,j}= p_1^i(T_j^j)-|M_{j,1} \cap M_T| +|M_{j,i+1-j} \cap M_T|.$$
Thus
\begin{equation}
    \label{eq:p}
x_{k+i-2,i-1} = p_{1}^{k+i-2}(T_{i-1}^{i-1}) 
-|M_{i-1,1}\cap M_T| +|M_{i-1,k}\cap M_T|, 
\end{equation}
\begin{equation}
    \label{eq:q}
x_{k+i-1,i} = p_{1}^{k+i-1}(T_{i}^{i}) 
-|M_{i,1}\cap M_T| +|M_{i,k}\cap M_T|. 
\end{equation}
Since $1< i \leq n$ and $ 1 \leq k \leq n-i+1,$ so by \eqref{eq:mik}, we have
$$M_{i-1,k}=\{ (k+i-1,i-1),(k+i,i-1),\dots, (n,i-1)\},$$
\begin{equation*}
 M_{i-1,k+1} = \left\{ 
    \begin{array}{ll}\{ (k+i,i-1),\dots, (n,i-1) \} & \quad \text{if} \quad 1 < i \leq n, 1 \leq k \leq n-i,  \\
    \emptyset & \quad \text{if} \quad 1 < i \leq n, k=n-i+1. 
    \end{array} \right.
\end{equation*}
Then the following is evident
\begin{equation}
\label{eq:r}
|M_{i-1,k} \cap M_T| = \left\{ 
    \begin{array}{ll}|M_{i-1,k+1}\cap M_T| +1 & \quad \text{if} \quad (k+i-1,i-1) \in M_{T},  \\
    |M_{i-1,k+1}\cap M_T| & \quad \text{elsewhere}. 
    \end{array} \right. 
\end{equation}
Applying \eqref{eq:r} in \eqref{eq:p} and assuming $c_{i,k}= p_{1}^{k+i-2}(T_{i-1}^{i-1})-|M_{i-1,1}\cap M_T| +|M_{i-1,k+1} \cap M_T| $, we get
\begin{equation}
\label{eq:U}
x_{k+i-2,i-1} = \left\{ 
    \begin{array}{ll}  
c_{i,k}+1 & \quad \text{if} \quad (k+i-1,i-1) \in M_{T},  \\
     c_{i,k}  & \quad \text{elsewhere}. 
    \end{array} \right. 
\end{equation}
Applying \eqref{eq:q} and \eqref{eq:U} in the last expression of $ \N_{i-1,k+1}^{\uparrow}-\N_{i,k}^{\uparrow},$ we get
\begin{equation*}
 \N_{i-1,k+1}^{\uparrow}(\Tilde{T})-\N_{i,k}^{\uparrow}(\Tilde{T}) = \left\{ 
    \begin{array}{ll}
        x_{k+i-2,i-1}-x_{k+i-1,i}-1 & \quad \text{if} \quad (k+i-1,i-1) \in M_{T},  \\
        x_{k+i-2,i-1}-x_{k+i-1,i} & \quad \text{elsewhere}. 
    \end{array}   \right. 
\end{equation*}
Since $SE_{k+i-1,i-1}(X_T)=x_{k+i-2,i-1}-x_{k+i-1,i},$ by Definition~\ref{def:MGT}, we have
\begin{equation*}
 SE_{k+i-1,i-1}(X_T) \geq \left\{ 
    \begin{array}{ll}
        1 & \quad \text{if} \quad (k+i-1,i-1) \in M_{T},  \\
        0 & \quad \text{elsewhere}. 
    \end{array}   \right. 
\end{equation*}
Thus we conclude $\N_{i-1,k+1}^{\uparrow}(\Tilde{T}) \geq \N_{i,k}^{\uparrow}(\Tilde{T})$ for $1< i \leq n, 1 \leq k \leq n+1-i$. This proves $\Tilde{T}$ is $\mu$-dominant (by \eqref{eq:uparrow}). Next we verify that $\wt(\Tilde{T})=\nu-\mu$. Using \eqref{eq:nik4}, we obtain $$\N_{i,1}^{\uparrow}(\Tilde{T})= \lambda_i + p_{1}^{i}(T_i^i)+
p_{i}^{i}(T_1^{i-1}).$$
Since $T$ is a semi-standard set-valued tableau, the entries $1,2,\dots,i-1$ cannot appear in $T_i,$ the $i^{th}$ row of $T$ from the top. Thus $p_{1}^{i}(T_i^i)=p_{i}^{i}(T_i^i)$. Hence we get
$$\N_{i,1}^{\uparrow}(\Tilde{T})= \lambda_i + p_{i}^{i}(T_i^i)+
p_{i}^{i}(T_1^{i-1}) =\lambda_i +p_i^i(T_1^i) = \nu_i.$$
The last equality in the above expression holds since $\wt(T)=\nu-\lambda$ and
$p_i^i(T_1^i)$ counts the total occurrences of the entry $i$ in $T_1,\dots,T_i,$ i.e., in $T$.
Also, using \eqref{eq:nik2},
$$\N_{i,1}^{\uparrow}(\Tilde{T})= \mu_i +\text{ the number of occurrences of }i \text{ in }\Tilde{T}_1,\Tilde{T}_{2},\dots,\Tilde{T}_n.$$
i.e.,
$$\N_{i,1}^{\uparrow}(\Tilde{T})= \mu_i +\text{ the number of occurrences of }i \text{ in }\Tilde{T}.$$
Therefore, the number of occurrences of the entry $i$ in $\Tilde{T}$ is $\nu_i - \mu_i$. Hence,
$\wt(\Tilde{T})=\nu-\mu$, which implies $\Tilde{T} \in \SVCT_{\nu-\mu}^{\mu}(\lambda)$.

Now we define the following map
\begin{equation}
    \label{eq:Gamma}
\Gamma : \SVT^{\lambda}_{\nu-\lambda}(\mu) \rightarrow \SVCT_{\nu-\mu}^{\mu}(\lambda) \text{ by }\Gamma(T)=\Omega(Y_T,M_T')=\Tilde{T}.  
\end{equation}
Our target now is to show $\Gamma$ is a bijection. First, we verify that $\Gamma$ is injective. Let $\Gamma(T)=\Gamma(T'),$ i.e., $ \Tilde{T}=\Tilde{T'}$. Also, let $\Upsilon^{-1}(T)=(X_T,M_T),\Upsilon^{-1}(T')=(X_{T'},M_{T'})$ and $\Omega^{-1}(\Tilde{T})=(Y_T,M'_T),\Omega^{-1}(\Tilde{T'})=(Y_{T'},M'_{T'})$, where $X_T=(x_{i,j})_{1 \leq j \leq i \leq n},X_{T'}=(x'_{i,j})_{1 \leq j \leq i \leq n}$ and $Y_T=(y_{i,j})_{1 \leq j \leq i \leq n},Y_{T'}=(y'_{i,j})_{1 \leq j \leq i \leq n}$. We will show that $x_{i,j}=x'_{i,j} $ for $1 \leq j \leq i \leq n$ and $M_{T}=M_{T'}.$ Since $M'_T=M'_{T'},$ it follows from \eqref{eq:m't} that $M_{T}=M_{T'}$.
Then it is clear that 
\begin{equation}
     \label{eq:mij}
    p_i^i(T_j^j)=x_{i,j}-x_{(i-1),j} + |\{(i,j)\} \cap M_T|. 
\end{equation}
By hypothesis, $y_{n-1,j}=y'_{n-1,j} \forall j \in \{ 1, \ldots, n-1\}$. Now, using \eqref{eq:yij}, we get $y_{n-1,j}=\N_{j+1,1}(T)$. Then using \eqref{eq:nik},
$$ \lambda_{j+1}+p_{j+1}^{j+1}(T_1^1)=\lambda_{j+1}+p_{j+1}^{j+1}(T'^1_1).$$
Applying \eqref{eq:mij},
$$\lambda_{j+1} + x_{j+1,1}-x_{j,1} + |\{(j+1,1)\} \cap M_T| =\lambda_{j+1} + x'_{j+1,1}-x'_{j,1} + |\{(j+1,1)\} \cap M_{T'}|.$$
Since $M_{T}=M_{T'},$
\begin{equation}
    \label{eq:injective}
   x_{j+1,1}-x_{j,1} = x'_{j+1,1}-x'_{j,1}.  
\end{equation}
For $j=1,$ we have 
$$ x_{2,1}-x_{1,1}=x'_{2,1}-x'_{1,1}.$$
Suppose that $T$ and $T'$ have weights $\wt(T)=(t_1,t_2,\dots,t_n)$ and $\wt(T')=(t'_1,t'_2,\dots,t'_n)$ respectively. Then $t_1=x_{1,1}$ and $t'_1=x'_{1,1}$. Since $\wt(T)=\wt(T'), x_{1,1}=x'_{1,1}.$ Therefore,
\begin{equation}
    \label{eq:j=1}
    x_{2,1}=x'_{2,1}.
\end{equation}
Putting $j=2$ in \eqref{eq:injective} we obtain,
$$ x_{3,1}-x_{2,1}=x'_{3,1}-x'_{2,1}.$$
Applying \eqref{eq:j=1}, we have $x_{3,1}=x'_{3,1}$. In the same way, $x_{n,1}=x'_{n,1}$. Similarly, for each fixed $k \in \{ 2, \dots, n-1\},$ $y_{n-k,j}=y'_{n-k,j} \forall j \in \{ 1, \ldots, n-k\}$ implies $x_{r,k}=x'_{r,k} \forall k \leq r \leq n$. Now for $n >1,$ let us define
$$A_n:=\{(n,1),(n,2),\dots,(n,n-1) \}.$$
We have 
$$
t_n=p_n^n(T_1^1)+\cdots + p_n^n(T_{n-1}^{n-1}) +p_n^n(T_n^n). 
$$
Then using \eqref{eq:mij}, we have
$$t_n=(x_{n,1}-x_{(n-1),1})+\cdots +(x_{n,n-1}-x_{(n-1),(n-1)}) +x_{n,n} +|A_n \cap M_T|,$$
$$t'_n=(x'_{n,1}-x'_{(n-1),1})+\cdots +(x'_{n,n-1}-x'_{(n-1),(n-1)}) +x'_{n,n} +|A_n \cap M_{T'}|.$$
Since $\wt(T)=\wt(T'),t_n=t'_n$. Also, we have proved that $x_{n,1}=x'_{n,1},\dots, x_{n,n-1}=x'_{n,n-1}$ and $x_{(n-1),1}=x'_{(n-1),1},\dots, x_{(n-1),(n-1)}=x'_{(n-1),(n-1)}$. Consequently, $ x_{n,n}=x'_{n,n}$ is obtained from the relations $t_n=t'_n$ and $M_T=M_{T'}$. Thus $X_T=X_{T'}$. Now combining the facts that $X_T=X_{T'}$ and $M_T=M_{T'}$, we conclude that $\Gamma$ is injective.

Now we prove that $\Gamma$ is surjective. Let $S \in \SVCT_{\nu-\mu}^{\mu}(\lambda)$ and $\Omega^{-1}(S)=(Z_{S},M_{S})$. Let $Z_{S}=(z_{i,j})_{1\leq j \leq i \leq n}$. For $1\leq j \leq i \leq n,$ we define the following
\begin{equation}
    \label{eq:dij}
     D_{i,j}:= \left\{
    \begin{array}{ll}  \bigl\{(n-j+1,1),(n-j+1,2),\dots, (n-j+1,i-j) \bigl\} &  \quad \text{if} \quad 1\leq j<i \leq n, \\
       \emptyset & \quad  \text{if} \quad 1\leq j=i \leq n.
     \end{array} \right.
\end{equation}
For $1 \leq j \leq n,$ we also define
\begin{equation}
    \label{eq:AB}
 z_{n-j,0}:=\nu_j. 
\end{equation}
Then we define another triangular array $Z_{S'}=(z'_{i,j})_{1 \leq j \leq i \leq n}$ by
\begin{equation}
\label{eq:z'_{i,j}}
    z'_{i,j}:=(z_{n-j,0} + z_{n-j,1} + \cdots + z_{n-j,i-j}) - (z_{n-j+1,1} +z_{n-j+1,2}+\cdots +z_{n-j+1,i-j+1})- |D_{i,j} \cap M_{S}|.
\end{equation}
In the above expression, $|D_{i,j} \cap M_{S}|$ denotes the cardinality of the set $D_{i,j} \cap M_{S}$. Then we show that $Z_{S'} \in \GT_{\mathbb{Z}}(\mu)$.
For $1 \leq j < i \leq n,$ we have
$$NE_{i,j}(Z_{S'})=z'_{i,j}-z'_{i-1,j}=(z_{n-j,i-j}-z_{n-j+1,i-j+1}) -|D_{i,j} \cap M_S| + |D_{i-1,j} \cap M_S|.$$
By definition, $D_{i,j}=D_{i-1,j} \cup \{(n-j+1,i-j)\}$. Thus
\begin{equation}
\label{eq:d i-1 to i}
 |D_{i-1,j} \cap M_S| = \left\{ 
    \begin{array}{ll}
        |D_{i,j} \cap M_S| -1 & \quad \text{if} \quad (n-j+1,i-j) \in M_{S},  \\
        |D_{i,j} \cap M_S| & \quad \text{elsewhere}. 
    \end{array}  \right. 
\end{equation}
Thus we obtain
\begin{equation}
    \label{eq:NE_{i,j}}
    NE_{i,j}(Z_{S'})=z'_{i,j}-z'_{i-1,j}=SE_{n-j+1,i-j}(Z_S) -|\{ (n-j+1,i-j) \cap M_S\}|.
\end{equation}
By Definition~\ref{def:MGT},
\begin{equation*}
 SE_{n-j+1,i-j}(Z_S) \geq \left\{ 
    \begin{array}{ll}
        1 & \quad \text{if} \quad (n-j+1,i-j) \in M_{S},  \\
        0 & \quad \text{elsewhere}. 
    \end{array}   \right. 
\end{equation*}
Thus we have
\begin{equation}
    \label{eq:B1}
NE_{i,j}(Z_{S'}) \geq 0.    
\end{equation}
Given $1 \leq j < i \leq n,$ we get
\begin{equation}
    \label{eq:SE_{i,j}(Z_S)}
     SE_{i,j}(Z_{S'})=z'_{i-1,j}-z'_{i,j+1}=(z_{n-j,0}- A_{i,j})-(z_{n-j-1,0}-B_{i,j}),
\end{equation}
where
\begin{equation}
    \label{eq:A_{i,j}}
A_{i,j}=(z_{n-j+1,1}-z_{n-j,1})+ \cdots  + (z_{n-j+1,i-j}-z_{n-j,i-j}) + |D_{i-1,j} \cap M_S|,  
\end{equation}
and 
\begin{equation}
    \label{eq:B_{i,j}}
B_{i,j}= (z_{n-j,1}-z_{n-j-1,1})+ \cdots + (z_{n-j,i-j-1}-z_{n-j-1,i-j-1}) + |D_{i,j+1}\cap M_S|.
\end{equation}
We recall that $S_{p,[q',q]}^{\uparrow}$ denotes the number of occurrences of the entry $p$ within the rows $S_{q'}^{\uparrow},S_{q'+1}^{\uparrow},\dots,S_q^{\uparrow},$ where $S_i^{\uparrow}$ denotes the $i^{th}$ row of $S$ indexed from the bottom. The semi-standard property of the set-valued contratableau $S$ implies that the entry $j$ must lie in $S_{1}^{\uparrow},S_{2}^{\uparrow},\dots,S_{n-j+1}^{\uparrow}.$ Thus $S_{j,[k,k]}^{\uparrow} =0 $ for $k > n-j+1.$ It is worth noting that $(n-j+1,n-j+1) \notin M_S$, as the condition $i >j$ holds for any $(i,j) \in M_S$ (see Definition~\ref{def:MGT}). Now  assuming $z_{n-j,n-j+1}:=0,$ we have for $1 \leq k \leq n-j+1$
\begin{equation}
    \label{eq:uparrow-2}
S_{j,[k,k]}^{\uparrow}= (z_{n-j+1,k}-z_{n-j,k}) + |\{(n+j-1,k)\} \cap M_S|. 
\end{equation}
As $S_{j,[1,i-j]}^{\uparrow}=S_{j,[1,1]}^{\uparrow}+\cdots+S_{j,[i-j,i-j]}^{\uparrow},$
it follows  from \eqref{eq:uparrow-2} that 
\begin{equation}
  \label{eq:spq}
S_{j,[1,i-j]}^{\uparrow}=(z_{n-j+1,1}-z_{n-j,1})+\cdots + (z_{n-j+1,i-j}-z_{n-j,i-j}) + |D_{i,j} \cap M_S|.
\end{equation}
Thus using \eqref{eq:d i-1 to i}, \eqref{eq:spq} in \eqref{eq:A_{i,j}} we get
\begin{equation*}
 A_{i,j} = \left\{ 
    \begin{array}{ll}
        S_{j,[1,i-j]}^{\uparrow} -1 & \quad \text{if} \quad (n-j+1,i-j) \in M_{S},  \\
        S_{j,[1,i-j]}^{\uparrow} & \quad \text{elsewhere}. 
    \end{array}  \right. 
\end{equation*}
Let $\wt(S)=(s_1,s_2,\dots,s_n).$ Since $\wt(S)=\nu-\mu,$ it follows that $s_j=\nu_j-\mu_j$. We also note that $s_j=S_{j,[1,i-j]}^{\uparrow}+S_{j,[i-j+1,n]}^{\uparrow}$ and $z_{n-j,0}=\nu_j$ (see \eqref{eq:AB}). Hence $\nu_j - S_{j,[1,i-j]}^{\uparrow}=\mu_j +S_{j,[i-j+1,n]}^{\uparrow} = \N^{\uparrow}_{j,i-j+1}(S)$ (by \eqref{eq:nik2}). Thus
\begin{equation}
  \label{eq:z-Aij}
z_{n-j,0}- A_{i,j} = \left\{ 
    \begin{array}{ll}
        \N^{\uparrow}_{j,i-j+1}(S) +1 & \quad \text{if} \quad (n-j+1,i-j) \in M_{S},  \\
        \N^{\uparrow}_{j,i-j+1}(S) & \quad \text{elsewhere}. 
    \end{array}  \right.
\end{equation}
Similarly, $B_{i,j}=S_{j+1,[1,i-j-1]}^{\uparrow}.$ Thus
\begin{equation}
  \label{eq:z-Bij}
z_{n-j-1,0}-B_{i,j} = \nu_{j+1}-S_{j+1,[1,i-j-1]}^{\uparrow}= \mu_{j+1} + S_{j+1,[i-j,n]}^{\uparrow} = \N^{\uparrow}_{j+1,i-j}(S) .
\end{equation}
Using \eqref{eq:z-Aij},\eqref{eq:z-Bij} in \eqref{eq:SE_{i,j}(Z_S)}, we have
\begin{equation*}
S_{i,j}(Z_{S'}) = \left\{ 
    \begin{array}{ll}
        \N^{\uparrow}_{j,i-j+1}(S)-\N^{\uparrow}_{j+1,i-j}(S)+1 & \quad \text{if} \quad (n-j+1,i-j) \in M_{S},  \\
        \N^{\uparrow}_{j,i-j+1}(S)-\N^{\uparrow}_{j+1,i-j}(S) & \quad \text{elsewhere}. 
    \end{array}  \right.
\end{equation*}
Since $S$ is $\mu$-dominant, $\N^{\uparrow}_{j,i-j+1}(S)\geq \N^{\uparrow}_{j+1,i-j}(S)$ (by \eqref{eq:uparrow}). Thus
\begin{equation}
  \label{eq:Sij(Z_S)}
S_{i,j}(Z_{S'}) \geq \left\{ 
    \begin{array}{ll}
         1 & \quad \text{if} \quad (n-j+1,i-j) \in M_{S},  \\
        0 & \quad \text{elsewhere}. 
    \end{array}  \right.
\end{equation}
By \eqref{eq:z'_{i,j}}, we have
$$z'_{n,j}=(z_{n-j,0} + z_{n-j,1} + \cdots + z_{n-j,n-j}) - (z_{n-j+1,1} +z_{n-j+1,2}+\cdots 
+z_{n-j+1,n-j+1})- |D_{n,j} \cap M_S| .$$
By isolating the term $z_{n-j,0}$ and grouping the remaining terms, this expression simplifies to:
$$ z'_{n,j}= z_{n-j,0}-\Bigl\{(z_{n-j+1,1}-z_{n-j,1})+\cdots + (z_{n-j+1,n-j}-z_{n-j,n-j}) +z_{n-j+1,n-j+1} +| D_{n,j} \cap M_S| \Bigl \}.$$
We recall that $\wt(S)=(s_1,s_2,\dots,s_n).$ Using \eqref{eq:uparrow-2} and the following fact 
$$s_j= S_{j,[1,1]}^{\uparrow}+\cdots+S_{j,[n-j,n-j]}^{\uparrow}+S_{j,[n-j+1,n-j+1]}^{\uparrow},$$
we have
$$ s_j= (z_{n-j+1,1}-z_{n-j,1}) +\cdots + (z_{n-j+1,n-j}-z_{n-j,n-j}) +z_{n-j+1,n-j+1} +| D_{n,j} \cap M_S|.$$
Also, $s_j=\nu_j-\mu_j$ (since $\wt(S)=\nu-\mu$) and $z_{n-j,0}=\nu_j$ (see \eqref{eq:AB}). Thus $$z'_{n,j}=z_{n-j,0}-s_j=\nu_j-(\nu_j-\mu_j)=\mu_j.$$
Therefore, by \eqref{eq:B1}, \eqref{eq:Sij(Z_S)} and the above equation, we conclude that $Z_{S'} \in \GT_{\mathbb{Z}}(\mu)$ and using \eqref{eq:Sij(Z_S)} we can also say that
$(Z_{S'},M_{S'}) \in \MGT_{\mathbb{Z}}(\mu),$ where 
\begin{equation}
    \label{eq:z}
 M_{S'}:=\{(i,j):(n-j+1,i-j) \in M_S \}.   
\end{equation}
Let $\Upsilon(Z_{S'},M_{S'})=S'$. Then we show that $S' \in \SVT^{\lambda}_{\nu-\lambda}(\mu)$ such that $\Gamma(S')=S$.

For each $1 \leq i \leq n , 1 \leq k \leq i,$ we recall from \eqref{eq:nik} that 
\begin{equation}
    \label{eq:Y}
\N_{i,0}(S')= \lambda_i; \N_{i,k}(S')= \lambda_i+ p_i^i({S'}_1^k)=\lambda_i + p_i^i({S'}_1^1)+p_i^i({S'}_2^2)+\cdots + p_i^i({S'}_k^k).
\end{equation}
Since $p_i^i({S'}_r^r)=z'_{i,r}-z'_{i-1,r} +|\{(i,r)\} \cap M_{S'}| $ for $1 \leq r \leq i.$ Here we define $z'_{i-1,i}:=0.$
Thus we obtain
$$\N_{i,k}(S')= \lambda_i + (z'_{i,1}-z'_{i-1,1}) + \dots  +(z'_{i,k}-z'_{i-1,k}) + |\{(i,1),\dots,(i,k) \} \cap M_{S'}|. $$
Using \eqref{eq:NE_{i,j}}, we obtain

$\N_{i,k}(S')= \lambda_i + SE_{n,i-1}(Z_S) + SE_{n-1,i-2}(Z_S) + \cdots + SE_{n-k+1,i-k}(Z_S) - |\{ (n,i-1),(n-1,i-2),\dots, (n-k+1,i-k)\} \cap M_{S}| + |\{(i,1),\dots,(i,k) \} \cap M_{S'}|.$

By \eqref{eq:z}, $(i,j) \in  M_{S'}$ if and only if $(n-j+1,i-j) \in M_S$. Thus 
$$|\{ (n,i-1),\dots, (n-k+1,i-k)\} \cap M_{S}| = |\{(i,1),\dots,(i,k) \} \cap M_{S'}|.$$
Therefore we have
$$\N_{i,k}(S')=  \lambda_i + SE_{n,i-1}(Z_S) + SE_{n-1,i-2}(Z_S) + \cdots + SE_{n-k+1,i-k}(Z_S).$$
Since $SE_{i,j}(Z_S)=z_{i-1,j}-z_{i,j+1},$ we get
$$\N_{i,k}(S') = \lambda_i +(z_{n-1,i-1}-z_{n,i}) + (z_{n-2,i-2} -z_{n-1,i-1}) + \cdots + (z_{n-k,i-k}-z_{n-k+1,i-k+1}).$$
Using $z_{n,i}=\lambda_i,$ and then simplifying the above expression, we have 
\begin{equation}
    \label{eq:X}
\N_{i,k}(S')=z_{n-k,i-k},
\end{equation}
for $0 \leq k \leq i.$
Thus for $ 1 \leq i <n , 0 \leq k \leq i,$ we have
$$ \N_{i,k}(S')-\N_{i+1,k+1}(S') = z_{n-k,i-k}-z_{n-k-1,i-k}= NE_{n-k,i-k}(Z_S) \geq 0 .$$
Hence by \eqref{eq:domi}, we conclude that $S'$ is $\lambda$-dominant.
Also, by \eqref{eq:Y}, \eqref{eq:X}, we have 
$$\N_{j,j}(S')=\lambda_i +p_j^j({S'}_1^j) = z_{n-j,0} =\nu_j .$$
Thus, $p_j^j({S'}_1^j) = \nu_j-\lambda_j.$ By the semi-standard property of $S',$ the entry $j$ is restricted to the top $j$ rows of $S'$. Hence if $\wt(S')=(s'_1,\dots,s'_n),$ then $s'_j=p_j^j({S'}_1^j)=\nu_j-\lambda_j$. Thus $\wt(S')=\nu-\lambda$. Also, we already proved that $S'$ is a $\lambda$-dominant semi-standard set-valued tableau of shape $\mu$. Therefore $S' \in \SVT^{\lambda}_{\nu-\lambda}(\mu)$.

Now if $\Gamma(S')=\Bar{S}$ and $\Omega(Y_{\Bar{S}},M_{\Bar{S}})=\Bar{S}$, where $Y_{\Bar{S}}=(\Bar{y}_{i,j})_{1 \leq j \leq i \leq n}$, then using \eqref{eq:yij} and \eqref{eq:X},
$$ \Bar{y}_{i,j}=\N_{n-i+j,n-i}(S')=z_{n-(n-i),(n-i+j)-(n-i)}=z_{i,j}.$$
Thus $Y_{\Bar{S}}=Z_S.$ Also, using \eqref{eq:m't}, $M_{\Bar{S}}=\{(n-j+1,i-j): (i,j) \in M_{S'}\}$.

Let $(n-j+1,i-j) \in M_{\Bar{S}}$. Then $(i,j) \in M_{S'}.$ So by \eqref{eq:z}, $(n-j+1,i-j) \in M_S$. Thus $M_{\Bar{S}} \subset M_S.$ Conversely, let $(i,j) \in M_S$ and $i_1=n-i+j+1, j_1=n-i+1$. Then $(n-j_1+1,i_1-j_1)=(i,j) \in M_S$. So $(i_1,j_1) \in M_{S'}$ (by \eqref{eq:z}). Hence $(n-j_1+1,i_1-j_1)=(i,j) \in M_{\Bar{S}}$. Thus $ M_S \subset M_{\Bar{S}},$ it follows that $M_S=M_{\Bar{S}}$. So $\Gamma(S')= \Bar{S}=\Omega(Y_{\Bar{S}},M_{\Bar{S}})=\Omega(Z_S,M_S)= S$. Thus we conclude that $\Gamma$ is surjective.
\end{proof}
\begin{example}
Let $\lambda=(3,2,1), \mu =(3,1), \nu =(4,4,3,2)$. Then $\SVT_{\nu-\lambda}^{\lambda}(\mu)$ contains the following two tableaux:
    $$ T=  \ytableausetup{mathmode,
notabloids,boxsize=2.4em}
  \begin{ytableau}
    1&2,\!3 &4 \\2,\!3,\!4   
  \end{ytableau} , S=  \ytableausetup{mathmode,
notabloids,boxsize=2.4em}
  \begin{ytableau}
    1& 2& 3,\!4\\2,\!3,\!4   
  \end{ytableau}.$$
Then $\SVCT_{\nu-\mu}^{\mu}(\lambda)$ contains the following two tableaux:
$$ \Tilde{T}=  \ytableausetup{mathmode,
notabloids,boxsize=2.4em}
  \begin{ytableau}
  \none & \none &2 \\
  \none & 1,\!2,\!3 &3 \\
   2,\!3 & 4 & 4    
  \end{ytableau} , \Tilde{S}=  \ytableausetup{mathmode,
notabloids,boxsize=2.4em}
  \begin{ytableau}
    \none & \none &1,\!2 \\
  \none & 2,\!3 &3 \\
   2,\!3 & 4 & 4     
  \end{ytableau}.$$
Let $\Upsilon^{-1}(T)=(X_{T},M_{T})$, where 
\begin{equation}
    \label{eq:ex-m_t}
   M_{T}=\{ (3,1), (3,2), (4,2) \} 
\end{equation} and 
\begin{center}
\begin{tikzpicture}[scale=0.8+0.2]
	\draw (1.5,1.5*1.732) node {$1$};
	\draw (1,1.732) node {$2$};
	\draw (2,1.732) node {$1$};
    \draw (0.5-1.5,0.5*1.732) node {$X_{T}=$};
	\draw (0.5,0.5*1.732) node {$2$};
	\draw (1.5,0.5*1.732) node {$1$};
	\draw (2.5,0.5*1.732) node {$0$};
	\draw (0,0) node {$3$};
	\draw (1,0) node {$1$};
	\draw (2,0) node {$0$};
	\draw (3,0) node {$0$};
\end{tikzpicture}.
\end{center}
Then under the map $\Gamma$ defined in \eqref{eq:Gamma}, the set-valued tableau $T$ corresponds to the element $ \Omega(Y_T,M'_T) \in \SVCT_{\nu-\mu}^{\mu}(\lambda)$ (see Proposition~\ref{proposition:contretableau}). Here, the Gelfand--Tsetlin pattern
$Y_T=(y_{i,j})_{1 \leq j \leq i \leq 4},$ is determined by \eqref{eq:yij}, i.e., by $y_{i,j}=\N_{4-i+j,4-i}(T)$. Thus we obtain
$$
\begin{tikzpicture}[scale=0.8+0.2]
	\draw (1.5,1.5*1.732) node {$2$};
	\draw (1,1.732) node {$3$};
	\draw (2,1.732) node {$2$};
    \draw (0.5-1.5,0.5*1.732) node {$Y_{T}=$};
	\draw (0.5,0.5*1.732) node {$3$};
	\draw (1.5,0.5*1.732) node {$2$};
	\draw (2.5,0.5*1.732) node {$1$};
	\draw (0,0) node {$3$};
	\draw (1,0) node {$2$};
	\draw (2,0) node {$1$};
	\draw (3,0) node {$0$};
\end{tikzpicture}.$$
Also, $M'_{T}$ is defined by via \eqref{eq:m't}:
$$M'_{T}=\{(5-j,i-j): (i,j) \in M_{T}\}.$$
Substituting the values from \eqref{eq:ex-m_t}, we obtain:
$$M'_{T}= \{ (4,2), (3,1), (3,2) \}.$$
Now we see that $\Omega(Y_{T},M'_{T})=\Tilde{T}  $ (by Proposition~\ref{proposition:contretableau}) as follows:
$$ 
\begin{tikzpicture}[scale=1.6]
  \draw (0,0)--(0,0.5)--(1,0.5)--(1,0)--(0,0);  
  \draw (0.5,0)--(0.5,0.5);
  \draw (0.25,0.25) node {$4$};
  \draw (0.75,0.25) node {$4$};
  \draw (0.5,-0.5) node {$Y^{(1)}$};
\end{tikzpicture}
\begin{tikzpicture}
    \node at (0,0)   {$\null$};
    \node at (0.5,0) {$\null$};
\end{tikzpicture}
\begin{tikzpicture}[scale=1.6]
  \draw (0,0)--(0,0.5)--(1,0.5)--(1,0)--(0,0);  
  \draw (0.5,0)--(0.5,1)--(1,1);
  \draw (1,0.5)--(1,1)--(1.5,1)--(1.5,0)--(1,0);
  \draw (1,0.5)--(1.5,0.5);
  \draw (0.25,0.25) node {$3$};
  \draw (0.75,0.25) node {$4$};
  \draw (1.25,0.25) node {$4$};
  \draw (0.75,0.75) node {$3$};
  \draw (1.25,0.75) node {$3$};
  \draw (0.5+0.3,-0.5) node {$Y^{(2)}$};
\end{tikzpicture}
\begin{tikzpicture}
    \node at (0,0)   {$\null$};
    \node at (0.5,0) {$\null$};
\end{tikzpicture}
\begin{tikzpicture}[scale=1.6]
  \draw (0,0)--(0,0.5)--(1,0.5)--(1,0)--(0,0);  
  \draw (0.5,0)--(0.5,0.5);
  \draw (1,0.5)--(1,1)--(1.5,1)--(1.5,0)--(1,0);
  \draw (1,0.5)--(1.5,0.5);
  \draw (0.5,0.5)--(0.5,1)--(1,1)--(1,1.5)--(1.5,1.5)--(1.5,1);
  \draw (0.25,0.25-0.02) node {$2,\!3$};
  \draw (0.75,0.25) node {$4$};
  \draw (1.25,0.25) node {$4$};
  \draw (1.25,0.75) node {$3$};
  \draw (0.75,0.75-0.02) node {$2,\!3$};
  \draw (1.25,1.25) node {$2$};
  \draw (0.5+0.3,-0.5) node {$Y^{(3)}$};
\end{tikzpicture}
\begin{tikzpicture}
    \node at (0,0)   {$\null$};
    \node at (0.5,0) {$\null$};
\end{tikzpicture}
\begin{tikzpicture}[scale=1.6]
  \draw (0,0)--(0,0.5)--(1,0.5)--(1,0)--(0,0);  
  \draw (0.5,0)--(0.5,0.5);
  \draw (1,0.5)--(1,1)--(1.5,1)--(1.5,0)--(1,0);
  \draw (1,0.5)--(1.5,0.5);
  \draw (0.5,0.5)--(0.5,1)--(1,1)--(1,1.5)--(1.5,1.5)--(1.5,1);
  \draw (0.25,0.25-0.02) node {$2,\!3$};
  \draw (0.75,0.25) node {$4$};
  \draw (1.25,0.25) node {$4$};
  \draw (1.25,0.75) node {$3$};
  \draw (0.75,0.75-0.02) node {$1,\!2,\!3$};
  \draw (1.25,1.25) node {$2$};
  \draw (0.5+0.3,-0.5) node {$Y^{(4)} = \Tilde{T}$};
\end{tikzpicture}.$$
Similarly, it can be checked that $\Gamma(S)=\Tilde{S}$. 
\end{example}
\subsection*{Acknowledgements}
The author expresses sincere gratitude to the referee for carefully reading the manuscript and for providing insightful comments that have immensely contributed to improving the standard of this manuscript. The author also acknowledges that the work was initiated during his tenure at TIFR Mumbai and concluded while he was at NISER Bhubaneswar. The manuscript was subsequently extensively revised after the author joined IIT Kanpur.

\end{document}